\documentclass[12pt,reqno]{amsart}

\usepackage[utf8]{inputenc}
\usepackage{graphicx}
\usepackage{amsmath,amssymb}
\usepackage{hyperref}
\usepackage{xcolor}
\usepackage{bm}
\usepackage{comment}
\usepackage{cite}
\usepackage{times}
\usepackage{tikz}

\numberwithin{equation}{section}

\allowdisplaybreaks

\newtheorem{theorem}{Theorem}[section]
\newtheorem*{theorem*}{Theorem}
\newtheorem{definition}[theorem]{Definition}

\newtheorem{remark}[theorem]{Remark}
\newtheorem{lemma}[theorem]{Lemma}
\newtheorem{proposition}[theorem]{Proposition}
\newtheorem{corollary}[theorem]{Corollary}

\def\O{\Omega}

\def\dist{{\rm dist}}
\def\diam{{\rm diam}}

\def\dx{{\,\rm d}x}
\def\dy{{\,\rm d}y}

\def\dhx{{\,\rm d}\hat{x}}
\def\dhy{{\,\rm d}\hat{y}}

\def\dim{{\rm dim}}

\def\hx{\hat{x}}
\def\hy{\hat{y}}
\def\u{{\bf u}}
\def\v{{\bf v}}
\def\r{{\bf r}}
\def\I{{\bf I}}

\newcommand{\R}{{\mathbb R}}

\title[The Fractional Korn Inequality on Uniform Domains]{The Fractional Korn Inequality on Uniform Domains and New Korn Inequalities for Truncated Seminorms}

\author[G.~Acosta]{Gabriel Acosta}
\address[G.~Acosta]{Departamento de Matem\'atica, FCEyN, Universidad de Buenos Aires / IMAS, CONICET, Buenos Aires, Argentina}
\email{gacosta@dm.uba.ar}

\author[I.~Drelichman]{Irene Drelichman}
\address[I.~Drelichman]{Centro de Matem\'atica de La Plata, Facultad de Cs. Exactas, Universidad Nacional de La Plata / CONICET, Argentina}
\email{idrelichman@mate.unlp.edu.ar}

\author[R.~Dur\'an]{Ricardo Dur\'an}
\address[R.~Dur\'an]{Departamento de Matem\'atica, FCEyN, Universidad de Buenos Aires / IMAS, CONICET, Buenos Aires, Argentina}
\email{rduran@dm.uba.ar}

\author[F.~L\'opez-Garc\'ia]{Fernando L\'opez-Garc\'ia}
\address[F.~L\'opez-Garc\'ia]{Department of Mathematics and Statistics, California State Polytechnic University Pomona\\ Pomona, CA 91768, US}
\email{fal@cpp.edu}

\author[I.~Ojea]{Ignacio Ojea}
\address[I.~Ojea]{Departamento de Matem\'atica, FCEyN, Universidad de Buenos Aires / IMAS, CONICET, Buenos Aires, Argentina}
\email{iojea@dm.uba.ar}

\thanks{Supported in part by PIP-2023, grant 11220220100246CO (G. Acosta, I. Drelichman, R. Dur\'an, I. Ojea) and by grant UBACyT 20020160100144BA (I. Drelichman, R. Dur\'an, I. Ojea)}

\begin{document}
\subjclass[2010]{26D10, 46E35, 46E40, 74B99.}
\maketitle

\begin{abstract}
We prove the so-called second case of the fractional Korn inequality for uniform domains.  We obtain this result as an application of a novel fractional Korn-type inequality  formulated in terms of truncated seminorms, which turns out to be valid for the broader class of John domains. We also obtain weighted estimates in which the weights are certain powers of the distance to the boundary that depend on the fractional exponent and the Assouad codimension of the boundary of the domain. 
\end{abstract}

\section{Introduction}

Let $\Omega \subset \mathbb{R}^n$ be a bounded domain and $\u: \Omega \to \mathbb{R}^n$ a vector field belonging to the Sobolev space $W^{1,p}(\Omega)^n$ with $1<p<\infty$. We denote with $\nabla \u$ the differential matrix of $\u$. The domain $\Omega$ is said to support the \emph{unconstrained} Korn inequality if
\begin{equation}
  \label{eq:kornG}
  \|\nabla \u\|_{L^p(\Omega)^{n\times n}} \leq C \left(\|\u\|_{L^p(\Omega)^n}+ \|\varepsilon(\u)\|_{L^p(\Omega)^{n\times n}}\right) \qquad  \forall \u\in W^{1,p}(\Omega)^n,
\end{equation}
where  $\varepsilon(\u)$ denotes the symmetric part of the gradient $\varepsilon(\u) = \frac{1}{2}(\nabla \u + (\nabla \u)^T)$ and $C$ is a constant depending only on $\Omega$ and $p$. Inequality \eqref{eq:kornG}  plays a central role in the mathematical theory of linear elasticity \cite{CiBook,MHbook},  in which $\Omega$ represents the volume occupied by an elastic body, $\u$ the displacement field and $\varepsilon(\u)$ the linearized strain tensor. It is well known that John domains support the \emph{unconstrained} Korn inequality \cite{ADbook}.

The Korn inequality is also known in the form
\begin{equation}
  \label{eq:korn}
  \|\nabla \u\|_{L^p(\Omega)^{n\times n}} \leq C \|\varepsilon(\u)\|_{L^p(\Omega)^{n\times n}},
\end{equation}
which is pivotal in proving the  existence and uniqueness of solutions of the linearized equations of elasticity under different boundary conditions.

Equation \eqref{eq:korn} establishes  the remarkable fact that the full gradient matrix of $\u$ can be bounded solely in terms of its symmetric part. However, that inequality  cannot hold for arbitrary vector fields in  $W^{1,p}(\Omega)^n$ due to the existence of functions in the kernel  of the linearized strain tensor $\varepsilon$ with a non-vanishing gradient.  This kernel, often called the space of \emph{infinitesimal rigid movements},   can be easily characterized as follows \cite{ADbook,CiBook,MHbook}
\begin{equation}
  \label{eq:RM}
  RM = \{ \u(x) = Ax+b \colon A\in \mathbb{R}^{n\times n}_{skew},  b \in \mathbb{R}^n \},
\end{equation}
where $ \mathbb{R}^{n\times n}_{skew}$ is the set of skew-symmetric matrices in $\mathbb{R}^{n\times n}$.

There are two classic cases, corresponding to homogeneous \emph{displacement} and \emph{traction} boundary conditions, in which \eqref{eq:korn} provides the coercivity of the linearized elasticity equations.  These cases are  known respectively as the \emph{first} and \emph{second} Korn inequalities. The \emph{first case}   states that \eqref{eq:korn} holds for vector fields with vanishing trace, that is, belonging to $W^{1,p}_0(\Omega)^n$. The  \emph{second case} establishes that \eqref{eq:korn} holds if the skew-symmetric part of $\nabla \u$ vanishes on average, i.e. for functions in  $W^{1,p}(\Omega)^n$ for which
\begin{equation}
  \int_\Omega \frac12\left(\nabla \u-(\nabla \u)^T\right) =0.
\end{equation}
The Korn inequality can also be found  in the form \cite{ADbook},
\begin{equation}
  \label{eq:korn_inf}
  \inf_{\r\in RM}\|\nabla (\u-\r)\|_{L^p(\Omega)^{n\times n}} \leq C \|\varepsilon(\u)\|_{L^p(\Omega)^{n \times n}},
\end{equation}
that is by withdrawing $RM$ from $W^{1,p}(\Omega)^n$, and sometimes, also referred to as  the \emph{second case} of the Korn inequality by some authors.

Standard compactness arguments show that  both cases  can be obtained from \eqref{eq:kornG}. However, it is not difficult to prove that the first case actually holds for \emph{any} bounded domain $\Omega$, as can be seen by extending  the fields in $W^{1,p}_0(\Omega)^n$ by zero and using the continuity properties of the Riesz transform \cite{KiOdbook} or even, in the case $p=2$, by means of elementary arguments \cite{ADbook}.   In strong contrast to the first case,  \eqref{eq:kornG} and \eqref{eq:korn_inf} are much more difficult to prove. Moreover, it is well know that they hold on John domains, being this class, in some cases, the most general class of domains for which these inequalities hold \cite{ADbook}.

The classical theory of elasticity, built upon differential operators, has failed to describe, at least without the introduction of appropriate ad-hoc assumptions,  phenomena involving singularities and discontinuities. This limitation has fostered the development of nonlocal elasticity models, such as \emph{peridynamics}  \cite{Silling2000,Silling2007}, which reformulates the laws of mechanics using integral, instead of differential, equations. In this context,  fractional Korn inequalities have gained interest within the mathematical community. In particular, in \cite{Harutyunyan2023a} the following substitutes of \eqref{eq:kornG} and \eqref{eq:korn_inf}  are studied
\begin{equation}
  \label{eq:KornGF}
  |\u|_{W^{s,p}(\Omega)^n} \le C \left( |\u|_{X^{s,p}(\Omega)^n} + \|\u\|_{L^p(\Omega)^n} \right)
\end{equation}
and
\begin{equation}
  \label{eq:Korn_infF}
  \inf_{r \in RM} |\u-\r|_{W^{s,p}(\Omega)^n} \le C |\u|_{X^{s,p}(\Omega)^n},
\end{equation}
where
\begin{equation}
  \label{eq:gagliardo_semi}
  |\u|_{W^{s,p}(\Omega)^n} := \left( \int_\Omega \int_\Omega \frac{|\u(x) - \u(y)|^p}{|x - y|^{n+sp}} \dx\dy \right)^{1/p}
\end{equation}
and
\begin{equation}
  \label{eq:meng_semi}
  |\u|_{X^{s,p}(\Omega)^n} := \left( \int_\Omega \int_\Omega \frac{|(\u(y) - \u(x)) \cdot (y - x)|^p}{|y - x|^{n+sp+p}} \dx\dy \right)^{1/p}.
\end{equation}
While \eqref{eq:gagliardo_semi} is the Gagliardo-Slobodekiij seminorm $|\u|_{W^{s,p}(\Omega)^n}$,  the less familiar expression \eqref{eq:meng_semi}  involves a nonlocal replacement for $\varepsilon(\u)$, as it is easy to see by a Taylor expansion of the vector field $\u$ around $x$
$$(\u(y) - \u(x)) \cdot (y - x) \approx (y-x)^T \varepsilon(\u)(x) (y-x).$$
Moreover, it is not difficult to prove \cite{TemamMiranville} that vector fields with vanishing $|\u|_{X^{s,p}(\Omega)^n}$ are those belonging to the set $RM$ given by \eqref{eq:RM}.

Let us briefly recall known results.  The fractional Korn inequality \eqref{eq:KornGF} has been studied for smooth domains and vector fields in the space $W^{s,p}_0(\Omega)^n$ \cite{Mengesha2022, Rut}.  In particular, the fractional \emph{first} case is studied in \cite{Harutyunyan2023b}, that is for vector fields in $W_0^{s,p}(\Omega)^n$, where it is shown to hold for $ps > 1$ and  bounded $C^1$ domains. The proof of this result relies on compactness arguments and a fractional Hardy-type inequality   that does not hold for $ps < 1$. Moreover,   in \cite{Harutyunyan2023b}, it is shown by means of a counterexample that  the fractional \emph{first} case of Korn inequality  does not hold if $ps < 1$ no matter how regular $\Omega$ is.  In \cite{Harutyunyan2023a} it is proved that $C^1$ domains, or even Lipschitz domains with a \emph{small enough} Lipschitz constant, support the fractional Korn inequalities \eqref{eq:KornGF} and \eqref{eq:Korn_infF}. 

In this work we show that \emph{uniform} domains support \eqref{eq:KornGF} and \eqref{eq:Korn_infF}, generalizing the above results, since the class of uniform domains  strictly contains that of Lipschitz domains and also some domains with fractal  boundaries such as the Koch snowflake. Moreover, they are, as proved in \cite{Jones1981}, \emph{extension} domains for classical Sobolev spaces.  We remark that, since the fractional Korn inequality proved in \cite{Harutyunyan2023a} is based on an extension argument, one could  try to extend that result to uniform domains using an extension technique similar to that developed in \cite{Jones1981}.

However, this is not the approach we use, mainly because classical Korn inequalities are known to hold on a much larger class of domains, namely, John domains, of which uniform domains are a proper subclass. But not all John domains are extension domains. Indeed, it is immediate to check that the unitary ball in $\R^2$ without the segment $\{(x,0):\, x\in(0,1)\}$ is a John domain but does not admit an extension theorem for classical Sobolev spaces. 

Therefore, we obtain the inequalities for uniform domains as a special case of an alternative form of the fractional Korn inequality given in terms of what we call \emph{truncated} seminorms, which is valid for John domains. We also obtain weighted estimates in which the weights are certain powers of the distance to the boundary. Specifically, any positive exponent is allowed, while negative exponents must be larger than an explicit parameter that depends on $s$ and the Assouad codimension of the boundary of the domain.

\section{A discrete Poincaré inequality on trees}
\label{section:discrete_poincare}
The proofs of the main results in this manuscript are based on a \emph{local-to-global} argument, in which the validity of the inequalities on cubes or other regular domains is extended to more complex domains, such as uniform or John domains. This argument requires the use of certain Poincar\'e-type inequalities on graphs, which we study in this section.

Let $G=(V,E)$ be a graph with a set of vertices $V$ and a set of edges $E$. We say that $G$ is a tree if it is connected and has no cycles. A \emph{rooted tree} is a tree where some node $t_0\in V$ is set as the root. On a rooted tree we can define a partial order $\preceq$ by saying that $r\preceq t$ if $r$ belongs to the unique path connecting $t$ with the root $t_0$. We also denote by $t_p$, the \emph{parent of $t$} as the first node in the path from $t$ to $t_0$.

Given a tree, we can consider sequences ${\bm b}=\{b_t\}_{t\in V}$ indexed over the set of vertices, with the partial order induced by $E$. A \emph{weight} is just a positive sequence over $V$. For $1\le p<\infty$ and a weight ${\bm \nu}$, we define $\ell^p(V,\nu)$ the space of sequences indexed over $V$ such that
\[\|{\bm b}\|_{p,\nu} := \left(\sum_{t\in V} |b_t|^p\nu_t\right)^\frac{1}{p}<\infty.\]

We denote by $\ell^p(V)$ the space of sequences with weight ${\bm \nu}\equiv 1$. 

In the following lemma, we establish a weighted Poincaré-type inequality on trees, whose proof is based on ideas developed in \cite{AO2016} for chains. In that work, the authors use an analogous discrete Poincaré inequality together with a local-to-global argument to obtain Korn inequalities on domains with a singularity.

\begin{lemma}\label{lemma:Poincare}
  Let $G=(V,E)$ be a rooted tree with root $t_0$ and ${\bm \nu}$, ${\bm \mu} \in \ell^1(V)$ two weights over $V$ such that the following Hardy-type inequality holds for every sequence ${\bm b}=\{b_t\}_{t\in V}$ over $V$:
  \begin{equation}\label{eq:discreteHardy}
    \left(\sum_{t\in V}\nu_t\left|\sum_{t_0\prec u\prec t}b_u\right|^p\right)^\frac{1}{p}\le C \left(\sum_{t\in V}\mu_t|b_t|^p\right)^\frac{1}{p}.
  \end{equation}
  Then, the following weighted Poincaré-type inequality holds
  \begin{equation}\label{eq:discretePoincare}
    \|{\bm b}-\bar{b}\|_{p,\nu} \le C_P \left(\sum_{t\succ t_0} |b_t-b_{t_p}|^p\mu_t\right)^\frac{1}{p},
  \end{equation}
  for all ${\bm b}$ indexed over $V$, where $\bar{b} = \frac{1}{\sum_{t\in V}\nu_t}\sum_{t\in V}b_t\nu_t$. The constant $C_P$ in \eqref{eq:discretePoincare} depends only on the constant $C$ in \eqref{eq:discreteHardy}
\end{lemma}
\begin{proof}
  Let us first observe that
  \[\|{\bm b}-  \bar{b}\|_{p,\nu} \le 2 \|{\bm b}-a\|_{p,\nu},\]
  for every constant $a$. Indeed by the triangular inequality
  \[\|{\bm b}-\bar{b}\|_{p,\nu} \le \|{\bm b}-a\|_{p,\nu} + \|\bar{b}-a\|_{p,\nu},\]
  and we only need to estimate the second term. Let $\nu(V) = \sum_{t\in V}\nu_t$. Then applying the H\"older inequality
  \begin{align*}
    \bar{b}-a &= \frac{1}{\nu(V)}\sum_{t\in V}b_t\nu_t - a = \frac{1}{\nu(V)}\sum_{t\in V}(b_t-a)\nu_t \\
    &\le \frac{1}{\nu(V)}\left(\sum_{t\in V} (b_t-a)^p\nu_t\right)^\frac{1}{p}\left(\sum_{t\in V}\nu_t\right)^\frac{1}{p'}\\
    &= \nu(V)^{-\frac{1}{p}} \|{\bm b}-a\|_{p,\nu}.
  \end{align*}
  Hence
  \[\|\bar{b}-a\|_{p,\nu} = |\bar{b}-a|\nu(V)^\frac{1}{p}\le \|{\bm b}-a\|_{p,\nu}.\]

  Now, we can obtain the Poincaré-type inequality by taking $a=b_{t_0}$ and applying the Hardy-type inequality. Indeed, 

  \begin{align*}
    \|{\bm b}-\bar{b}\|_{p,\nu}^p &\le 2^p\|{\bm b}-b_{t_0}\|_{p,\nu}^p = 2^p\sum_{t\in V} |b_t-b_{t_0}|^p \nu_t \\
    &= 2^p \sum_{t\in V}\left|\sum_{t_0\prec u\preceq t} b_{u}-b_{u_p} \right|^p \nu_t,
    \intertext{where in the inner summation we sum over all the vertices in the path from $t_0$ to $t$. Now by \eqref{eq:discreteHardy}}
    \|{\bm b}-\bar{b}\|_{p,\nu}^p &\le C \sum_{t\in V}|b_t-b_{t_p}|^p\mu_t.
  \end{align*}
\end{proof}

In the theorem below, we exhibit a sufficient condition on a weighted rooted tree that implies the validy of condition \eqref{eq:discreteHardy}. This result was proved in \cite{LGO-2022}.
\begin{theorem}\label{thm:suff cond}
  Let $G=(V,E)$ be a rooted tree with root $t_0$ and $\{\mu_t\}_{t\in V}$ and $\{\nu_t\}_{t\in V}$ two weights over $V$. We consider $1<p<\infty$ and $p'$ its H\"older conjugate, $\frac{1}{p}+\frac{1}{p'}=1$. Suppose that there is some $\theta>1$ such that:
  \begin{equation}\label{eq:suffcond}
    C_{tree}:=\sup_{t\in V}\left(\sum_{t_0\prec u\preceq t}\mu_u^{-\frac{p'}{p}}\right)^\frac{1}{\theta p'}\left(\sum_{u\succeq t}\nu_u\left(\sum_{t_0\prec v\prec u}\mu_v^{-\frac{p'}{p}}\right)^{\frac{p}{p'}(1-\frac{1}{\theta})}\right)^\frac{1}{p}<\infty.
  \end{equation}
  Then, inequality \eqref{eq:discreteHardy} holds for every sequence ${\bm b} \in \ell^p(V,\mu)$. Moreover the constant in \eqref{eq:discreteHardy} satisfies $C\le \left(\frac{\theta}{\theta-1}\right)^\frac{1}{p'}C_{tree}$.
\end{theorem}

\begin{remark}
In the manuscript \cite{LGR-2025}, the authors introduce a condition on weighted graphs with ${\bm \nu}={\bm \mu}$, similar to the one appearing in the upcoming chapter in Lemma \ref{lemma:treeJohn}, which implies the validity of a discrete Poincar\'e-type inequality analogous to \eqref{eq:discretePoincare}. 
\end{remark}

\section{Domains and Trees}\label{section:domains_trees}
Now we turn our attention to the classes of domains we will be working with. We begin by recalling some definitions.
\begin{definition}[Uniform Domain (see for example \cite{Jones1981}]
    A domain $D \subset \mathbb{R}^n$ is said to be \textbf{uniform} if there exists a constant $c \geq 1$ such that for any two points $x_1, x_2 \in D$, there is a rectifiable arc $\gamma \subset D$ joining them such that:
    \begin{enumerate}
      \item $\ell(\gamma) \leq c |x_1 - x_2|$, and
      \item $ \min \{ \ell(\gamma(x_1, x)), \ell(\gamma(x, x_2)) \} \leq c  \, d(x, \partial D)$ for all $x \in \gamma$,
    \end{enumerate}
    where $\ell(\gamma)$ is the arc length, $\gamma(x_i, x)$ is the sub-arc of $\gamma$ between $x_i$ and $x$, and $d(x, \partial D)$ is the  distance from $x$ to the boundary of $D$.
  \end{definition}

  \begin{definition}[John Domain (see for example \cite{ADbook}]
      A bounded domain $D \subset \mathbb{R}^n$ is called a \textbf{John domain} if there exists a constant $C \geq 1$ and a distinguished point $x_0 \in D$ such that every point $x \in D$ can be joined to $x_0$ by a rectifiable curve $\gamma: [0, \ell] \to D$, parameterized by its arc length, such that $\gamma(0) = x$, $\gamma(\ell) = x_0$, and
      \begin{equation}
        d(\gamma(t), \partial D) \geq \frac{1}{C} t
      \end{equation}
      for all $t \in [0, \ell]$, where $d(\cdot, \partial D)$ denotes the   distance to the boundary of $D$.
    \end{definition}

    There are several equivalent definitions of John domains. For our local-to-global argument we will use a definition based on trees of Whitney cubes.

    Let us recall that a Whitney decomposition of a bounded domain $\Omega\subset\R^n$ is a collection $\{Q_t\}_{t\in\Gamma}$ of closed dyadic cubes whose interiors are pairwise disjoint, such that the edge-length of $Q_t$ is proportional to its distance to $\partial\Omega$ and two neighboring cubes are proportional in size. For example, we can take
    \begin{enumerate}
      \item $\O=\bigcup_{t\in\Gamma}Q_t$,
      \item $3 \diam(Q_t) \leq d(Q_t,\partial\Omega) \leq 8\diam(Q_t)$,
      \item $\frac{1}{2}\diam(Q_u)\leq \diam(Q_t)\leq 2\diam(Q_u)$, if $Q_u\cap Q_t\neq \emptyset$,
    \end{enumerate}
    where $\diam(Q_t)$ stands for the diameter of the cube $Q_t$. Two distinct cubes $Q_u$ and $Q_t$ with $Q_u\cap Q_t\neq \emptyset$ are called {\it neighbors}. Notice that two neighbors may have an intersection with dimension less than $n-1$. For instance, they may intersect at a single point. We say that $Q_u$ and $Q_t$ are \mbox{$(n-1)$}-neighbors if $Q_u\cap Q_t$ is a $n-1$ dimensional face. This kind of covering exists for any proper open subset of $\R^n$ (see \cite{Stein1970} for details).

    The set of indices $\Gamma$ is numerable. It can be endowed with a tree structure, where the vertices of the tree are identified with the cubes and each edge connects two $(n-1)$-neighbouring cubes. However, only some $(n-1)$-neighbours induce an edge in the tree. Edges are added in a way such that all cubes are connected, but no cycles are formed. Another way of thinking about this tree construction is to consider the complete graph where all vertices (cubes) and all edges are present and then extracting a \emph{spanning tree} from this graph. A tree thus constructed, with a distinguished node as root, can be identified with a partial order along the set $\Gamma$, as we did in Section \ref{section:discrete_poincare}. Figure \ref{fig:whitney_tree_shadow} shows a Whitney decomposition of a domain, and indicates a possible tree structure built upon it.

    For every vertex $t\in \Gamma$, we define \emph{the shadow} of the cube $Q_t$ as the set:
    \begin{equation}\label{eq:shadow}
      W_t := \cup_{u\succeq t}Q_u.
    \end{equation}

    In Figure \ref{fig:whitney_tree_shadow} the shadow of the node $t$, corresponding to the cube $Q_t$, is marked along the tree.

    \begin{figure}[h]
      \centering
      \includegraphics[width=0.5\textwidth]{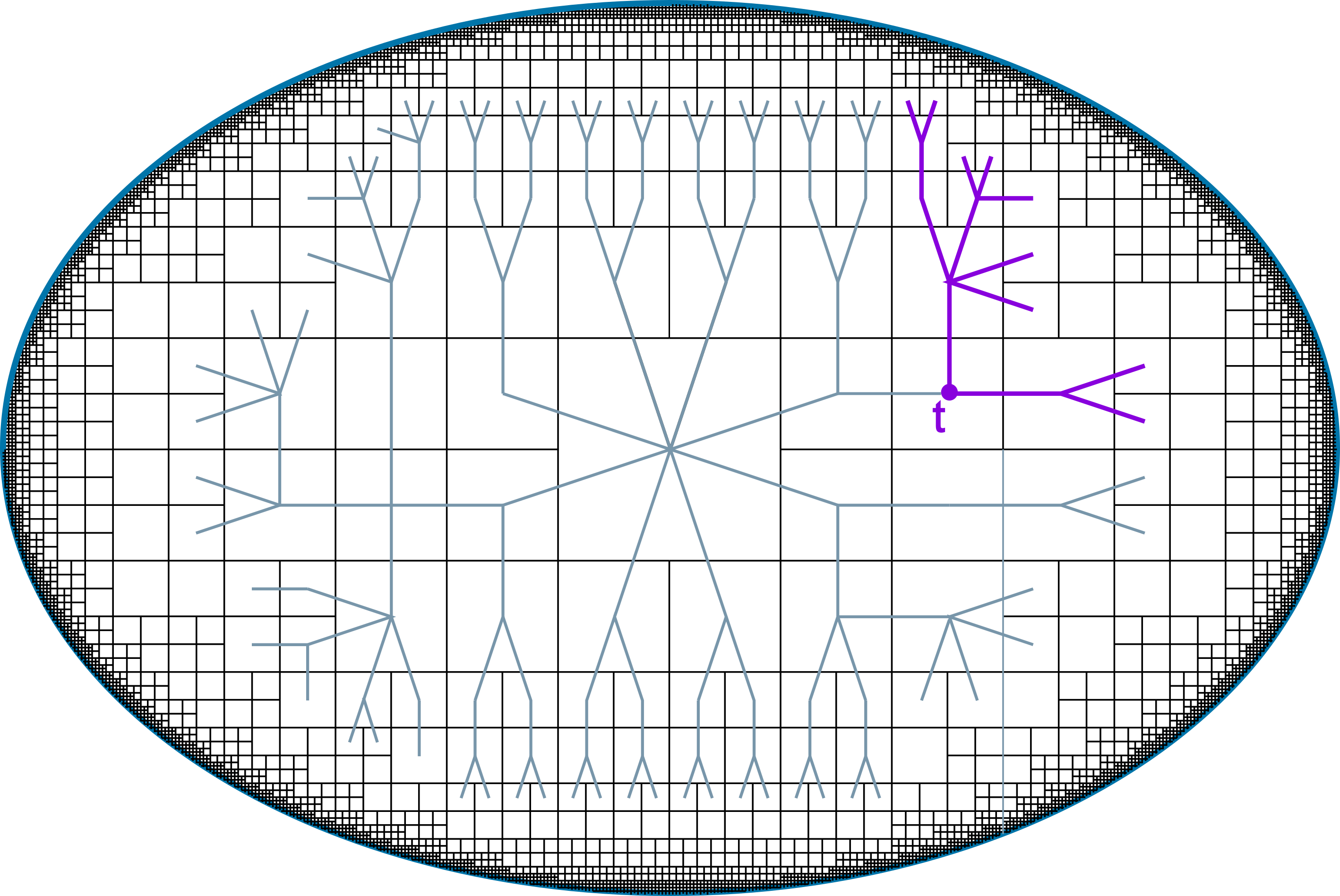}
      \caption{A Whitney decomposition, with a possible underlying tree-structure. The \emph{shadow} of the cube $Q_t$ is sketched.}
      \label{fig:whitney_tree_shadow}

    \end{figure}

    Naturally, there are infinitely many tree structures satisfying the description above for the same Whitney decomposition. The idea is to build a tree that contains relevant information about the domain geometry. The following result provides such a tree for John domains. It was proven in \cite[Lemma 5.2]{LG-2018}.

    \begin{lemma}\label{lemma:treeJohn} A bounded domain $\O$ in $\R^n$ is a John domain if and only if given a Whitney decomposition $\{Q_t\}_{t\in\Gamma}$ of $\O$, there exist a tree structure for the set of indices $\Gamma$ and a constant $K>1$
      that satisfy
      \begin{align}\label{Boman tree}
        Q_u\subseteq KQ_t,
      \end{align}
      for any $u,t\in\Gamma$ with $u\succeq t$. In other words, the shadow of
      $Q_t$ is contained in the expanded set $KQ_t$. Moreover, the intersection of the cubes associated to adjacent indices, $Q_t$ and $Q_{t_p}$, is an $n-1$ dimensional face of one of these  cubes.
    \end{lemma}

    In Section \ref{section:korn_john} we prove the validity of the discrete Poincaré inequality \eqref{eq:discretePoincare} on the tree given by Lemma \ref{lemma:treeJohn}. This allows us to extend the Korn inequality from cubes to a variation of the same inequality on John domains. 
    
    However, the Korn inequality obtained in \cite[Theorem 1.3]{Harutyunyan2023a} is not proved in cubes, but rather in $\mathcal{C}^1$ domains (or Lipschitz domains with a \emph{sufficiently small constant}). While proving the same result for cubes is possible using arguments analogous to the ones in that paper, doing so would require extensive rewriting of most of the proofs. To circumvent this limitation we introduce \emph{smoothened cubes} as a technical device.

      Let us define
      \begin{equation}\label{eq:hatQ}
        \hat{Q}:=\left\{x\in \R^n\colon \|x\|_\infty\leq 1\right\},
      \end{equation}
      the cube with edge-length $2$ centered at the origin. Observe that a cube $Q$ with edges parallels to the coordinate axes, edge-length $\ell_Q$ and center $q$  can be written as:
      \[Q = \tfrac{\ell_Q}{2} \hat{Q} + q.\]

      In a similar way, we define the \emph{smoothened cube}
      \begin{equation}\label{def:smothen_reference_cube}
        \hat{U}:=\{x\in \R^n\colon \|x\|_N\leq 1\},
      \end{equation}
      for some $N<\infty$ to be defined later.

      We will work with smoothened cubes of the form:
      \begin{equation}\label{eq:general_smoothen}
        U = \tfrac{\ell_U}{2}\hat{U} + q.
      \end{equation}
      By an abuse of notation, we say that $\ell_U$ is the edge-length of $U$.

     In the next section, we prove some results regarding skew-symmetric matrices on smoothened cubes.

    \section{Some preliminary results on skew-symmetric matrices}

    Let $\Omega\subset \R^n$ a bounded domain and $0<s<1$. For $\u\in W^{s,p}(\Omega)^n$, $\v\in W^{s,p'}(\Omega)^n$, and $E\subset \Omega$ we define the bilinear form
    \[\langle \u,\v \rangle_{E\times E} = \int_E\int_E \frac{(\u(x)-\u(y))\cdot(\v(x)-\v(y))}{|x-y|^{n+2s}} \dy\dx.\]

    Note that $n+2s = \frac{n}{p}+s+\frac{n}{p'}+s$, so the H\"older inequality yields $\langle \u,\v\rangle \le  |\u|_{W^{s,p}(\Omega)^n}|\v|_{W^{s,p'}(\Omega)^n}$, which gives the continuity of $\langle,\rangle:W^{s,p}(E)^n\times W^{s,p'}(E)^n\to \R$.

    When the domains of integration are both the whole domain $\Omega$ we simplify the notation by writing $\langle \u,\v \rangle$ instead of $\langle \u,\v \rangle_{\Omega\times\Omega}$. We also define $I_{ij}\in\R^{n\times n}$ the matrix such that:
    \[(I_{ij})_{km} = \left\{
        \begin{array}{cl} 1 & \textrm{ if } k=i,m=j \\
          -1 & \textrm{ if } k=j,m=i\\
          0  & \textrm{ elsewhere}.
      \end{array}\right.\]

      Note that $\{I_{ij}\}_{1\le i<j\le n}$ forms a basis of the skew-symmetric matrices of $n\times n$. We denote $\I_{ij}$ the linear transform $\I_{ij}(x) = I_{ij}x$.

      \begin{lemma}\label{lemma:orthogonality} If we denote by $\delta_{ik}$ a Kronecker delta we have that
        \[\langle \I_{ij},\I_{km} \rangle = |\I_{ij}|_{W^{2,s}(\Omega)^n}^2 \delta_{ik}\delta_{km}.\]
      \end{lemma}
      \begin{proof}
        Let  $e_i$ be the canonical vector that has $1$ in the $i$-th coordinate and $0$ elsewhere. Then, $I_{ij}x = x_je_i-x_ie_j$. Hence, if $i\neq k$ and $j\neq m$, $I_{ij}(x-y)\cdot I_{km}(x-y) = 0$ and, consequently, $\langle I_{ij},I_{km}\rangle = 0$.

        On the other hand, if $i=k$ and $j\neq m$, then $I_{ij}(x-y)\cdot I_{im}(x-y) = (x_j-y_j)(x_m-y_m)$, which yields
        \[\langle \I_{ij},\I_{im} \rangle = \int_\Omega\int_\Omega \frac{(x_j-y_j)(x_m-y_m)}{|x-y|^{n+2s}} \dy\dx.\]
        Now, we can make the change of variables that interchanges $x_j$ with $y_j$ which gives
        \[\langle \I_{ij},\I_{im} \rangle = \int_\Omega\int_\Omega \frac{(y_j-x_j)(x_m-y_m)}{|x-y|^{n+2s}} \dy\dx =-\langle \I_{ij},\I_{im} \rangle,\]
        and, hence, $\langle \I_{ij},\I_{im} \rangle = 0$.
        The cases $j=m\land i\neq k$, $i=m\land j\neq k$ and $j=k\land i\neq m$ are similar.

        Finally, if $k=i$, $m=j$ we have that  $\langle \I_{ij},\I_{ij} \rangle = |\I_{ij}|_{W^{2,s}(\Omega)}^2$.
      \end{proof}

      \begin{lemma}\label{lemma:seminormIijcubes}
        If $U$ is a smoothened cube with edge length $\ell_U$, as defined by \eqref{eq:general_smoothen}, then
        \[|\I_{ij}|_{W^{s,p}(U)^n} = \hat{C}\ell_U^{\frac{n}{p}+1-s},\]
        where $\hat{C}$ is a constant depending on $n,p$ and $s$ but not on $i,j$ or $\ell_U$.
      \end{lemma}
      \begin{proof}
        For $x,y \in U$, we make the change of variables $\hat{x} = \ell_U^{-1}(x-q)$, $\hat{y}=\ell_U^{-1}(y-q)$. Notice that this maps $U$ onto $\hat{U}$. Then:
        \begin{align*}
          |\I_{i,j}|_{W^{s,p}(U)^n}^p &= \int_U\int_U \frac{|I_{ij}(x-y)|^p }{|x-y|^{n+ps}}\dy\dx \\
          &= \int_U\int_U \frac{|(x_i-y_i)^2 + (x_j-y_j)^2|^\frac{p}{2}}{|x-y|^{n+ps}}\dy\dx \\
          &= \int_{\hat{U}}\int_{\hat{U}} \frac{\ell_U^p [(\hat{x}_i-\hat{y}_i)^2 + (\hat{x}_j-\hat{y}_j)^2]^\frac{p}{2}}{\ell_U^{n+ps}|\hat{x}-\hat{y}|^{n+ps}} \ell_U^{2n}\dhy\dhx\\
          &= \ell_U^{n+p-ps}\int_{\hat{U}}\int_{\hat{U}} \frac{[(\hat{x}_i-\hat{y}_i)^2 + (\hat{x}_j-\hat{y}_j)^2]^\frac{p}{2}}{|\hat{x}-\hat{y}|^{n+ps}} \dhy\dhx\\
          &= \ell_U^{n+p-ps}|\I_{ij}|_{W^{s,p}(\hat{U})^n}^p.
        \end{align*}
        By the symmetry of $\hat{U}$ and taking into account that $1\le i< j\le n$, it is clear that we can change, for example, the index $j$ by any other index $k>i$ and the value of the integral does not change. Consequently we can write $\hat{C} = |I_{ij}|_{W^{s,p}(\hat{U})^n}$ which is a constant independent of $i,j$.
      \end{proof}

      \begin{corollary}\label{corol:Iijcubes}
        If $U$ is a smoothened cube as in \eqref{eq:general_smoothen}, then
        \[\langle \I_{ij},\I_{km} \rangle_{U\times U} = \hat{C}\ell_U^{n+2-2s}\delta_{ik}\delta_{jm},\]
        where $\hat{C}$ is a constant depending on $n,p$ and $s$ but not on $i,j$ or $\ell_U$.
      \end{corollary}
      \begin{proof}
        The result follows by combining Lemmas \ref{lemma:orthogonality} and \ref{lemma:seminormIijcubes} with $p=2$.
      \end{proof}

      Now, for $\u \in W^{s,p}(\Omega)^n$ we define $\Pi_\Omega \u$ as the projection of the $s$ gradient of $\u$ onto the skew-symmetric matrices. Namely
      \[\Pi_\Omega\u = \sum_{1\le i<j\le n}u_{ij}\I_{ij},\]
      such that $\langle \u-\Pi_\Omega\u,\I_{ij}\rangle = 0$ for every $1\le i<j\le n$. Given the orthogonality of the linear transformations $\I_{ij}$ with respect to $\langle,\rangle$ it is easy to check that
      \[u_{ij} = |\I_{ij}|^{-2}_{W^{s,2}(\Omega)^n}\langle \u,\I_{ij}\rangle. \]

      \begin{proposition}
        Let $\Omega\subset \R^n$ a bounded domain, and $r$ and $R$ be the interior and exterior diameters of $\Omega$ respectively. Then
        \[|\u-\Pi_\Omega\u|_{W^{s,p}(\Omega)^n} \le C \left(\tfrac{R}{r}\right)^{n+2-2s}|\u-\r|_{W^{s,p}(\Omega)^n},\]
        for every $\r \in RM$.
      \end{proposition}
      \begin{proof}
        By the triangle inequality we have
        \[|\u-\Pi_\Omega\u|_{W^{s,p}(\Omega)^n} \le |\u-\r|_{W^{s,p}(\Omega)^n} + |\r-\Pi_\Omega\u|_{W^{s,p}(\Omega)^n} = I + II.\]
        We only need to estimate $II$. Thus, let us write $\r(x)=Bx + b$, for some skew-symmetric matrix $B\in \R^{n\times n}$  and $b\in\R^n$. Then, $B=\sum_{1\le i< j\le n} b_{ij}I_{ij}$ and $\Pi_\Omega\u = \sum_{1\le i<j\le n}u_{ij}I_{ij}$. Therefore, we have
        \begin{align*}
          II^p &= \int_\Omega\int_\Omega\frac{|B(x-y) -\Pi_\Omega\u (x-y)|^p}{|x-y|^{n+ps}}\dy\dx \\
          &\le C \sum_{1\le i <j \le n} \int_\Omega\int_\Omega\frac{|b_{ij}-u_{ij}|^p|I_{ij} (x-y)|^p}{|x-y|^{n+ps}}\dy\dx.
        \end{align*}
        Now, we have that
        \[b_{ij} = |\I_{ij}|_{W^{s,2}(\Omega)^n}^{-2} \langle \r,\I_{ij}\rangle,\quad\quad u_{ij} = |\I_{ij}|_{W^{s,2}(\Omega)^n}^{-2} \langle \u,\I_{ij}\rangle,\]
        which along with the H\"older inequality gives
        \[|b_{ij}-u_{ij}| = |\I_{ij}|_{W^{s,2}(\Omega)^n}^{-2} \langle \r-\u,\mathcal{\I}_{ij}\rangle \le |\I_{ij}|_{W^{s,2}(\Omega)^n}^{-2}|\r-\u|_{W^{s,p}(\Omega)^n}|\mathcal{\I}_{ij}|_{W^{s,p'}(\Omega)^n}.\]
        Inserting this in the previous estimate we obtain
        \begin{align*}
          II^p &\le C \sum_{1\le i<j\le n}|\I_{ij}|_{W^{s,2}(\Omega)^n}^{-2p}|\I_{ij}|_{W^{s,p'}(\Omega)^n}^p|\r-\u|_{W^{s,p}(\Omega)^n}^p\int_\Omega\int_\Omega\frac{|I_{ij} (x-y)|^p}{|x-y|^{n+ps}}\dy\dx \\
          &= C \sum_{1\le i<j\le n}|\I_{ij}|_{W^{s,2}(\Omega)^n}^{-2p}|\I_{ij}|_{W^{s,p'}(\Omega)^n}^p|\I_{ij}|_{W^{s,p}(\Omega)^n}^p|\r-\u|_{W^{s,p}(\Omega)^n}^p.
        \end{align*}
        In order to estimate the semi-norms of $\I_{ij}$ we change the domain of integration and use equivalence of norms in the (finite dimensional) space of skew-symmetric matrices in $\R^{n\times n}$. For example, consider a ball $B(z,R)$ with center $z$ and radius $R$ such that $\Omega\subset B(z,R)$. Then:
        \begin{align*}
          |\I_{ij}|_{W^{s,p}(\Omega)^n}^p &\le |\I_{ij}|_{W^{s,p}(B(z,R))^n}^p = \int_{B_R}\int_{B_R} \frac{|I_{ij}(x-y)|^p}{|x-y|^{n+ps}} \dy\dx.
        \end{align*}
        Changing variables to $\tilde{x} = R^{-1}(x-z)$, $\tilde{y}=R^{-1}(y-z)$
        \begin{align*}
          |\I_{ij}|_{W^{s,p}(\Omega)^n}^p &\le \int_{B(0,1)}\int_{B(0,1)}\frac{R^p |I_{ij}(\tilde{x}-\tilde{y})|^p}{R^{n+ps}|\tilde{x}-\tilde{y}|^{n+ps}}R^{2n}{\rm d}\tilde{x}{\rm d}\tilde{y} \\
          &= R^{n+p-ps} \int_{B(0,1)}\int_{B(0,1)}\frac{|I_{ij}(\tilde{x}-\tilde{y})|^p}{|\tilde{x}-\tilde{y}|^{n+ps}}{\rm d}\tilde{x}{\rm d}\tilde{y} \\
          &= R^{n+p-ps} |\I_{ij}|^p_{W^{s,p}(B(0,1))^n}\\
          &\le C R^{n+p-ps} |\I_{ij}|^p_{W^{s,2}(B(0,1))^n},
        \end{align*}
        where in the last step we applied the equivalence of norms in the space of skew-symmetric matrices. Similarly, we obtain that $|\I_{ij}|_{W^{s,p'}(\Omega)^n}^p\le C\left(R^{\frac{n}{p'}+1-s}|\I_{ij}|_{W^{s,2}(B(0,1))^n}\right)^p$.

        On the other hand we can take $B(w,r)$ to be a ball contained in $\Omega$ with center $w$ and radius $r$. Then, $|\I_{ij}|_{W^{s,2}(\Omega)^n} \ge |\I_{ij}|_{W^{s,2}(B(w,r))^n}$ and applying the same change of variables we have
        \[|\I_{ij}|_{W^{s,2}(B(w,r))^n}^2 = r^{n+2-2s}|\I_{ij}|^2_{W^{s,2}(B(0,1))^n},\]
        and, consequently,
        \[|\I_{ij}|_{W^{s,2}(\Omega)^n}^{-2p} \le r^{-np-2p+2ps}|\I_{ij}|^2_{W^{s,2}(B(0,1))^n}.\]

        All these estimates together give
        \begin{align*}|\I_{ij}|_{W^{s,2}(\Omega)^n}^{-2p}|\I_{ij}|_{W^{s,p'}(\Omega)^n}^p|\I_{ij}|_{W^{s,p}(\Omega)^n}^p &\le C R^{n+n\frac{p}{p'}+2p-2ps}r^{-np-2p+2ps}|\I_{ij}|^2_{W^{s,2}(B(0,1))^n}\\
          &= C\left(\tfrac{R}{r}\right)^{(n+2-2s)p}|\I_{ij}|^2_{W^{s,2}(B(0,1))^n}.
        \end{align*}

        Inserting this in the estimate for $II^p$ and taking into account that $\sum_{1\le i<j\le n}|\I_{ij}|^2_{W^{2,s}(B(0,1))^n}$ is a fixed constant, the result follows.
      \end{proof}

      \section{Weighted truncated fractional Korn inequalities on John domains}\label{section:korn_john}

      In this section, we study certain versions of the Korn inequality on truncated fractional spaces and John domains. First, let us briefly recall that this class contains domains for which $W^{1,p}(\O)^n \not\subset W^{s,p}(\O)^n$. Take, for example,  $\O=[-1,1]^2\setminus\{(x,0): x\in[0,1]\}$ and let $\u\in W^{1,\infty}(\O)^2$ be a vector field such that $\u\equiv (1,1)$ in $[\frac14,\frac12]\times (0,\frac12]$ and $\u\equiv (0,0)$ in $[\frac14,\frac12]\times [-\frac12,0)$. Then, clearly $\u\in W^{1,p}(\O)^2$ for every $1\le p<\infty$, whereas $\u\not\in W^{s,p}(\O)^2$ whenever $sp>1$, because functions in this space cannot have jumps. 

       One way to overcome this drawback is to introduce the truncated Gagliardo seminorm 
      \begin{equation}
        \label{eq:rest_gagliardo_semi}
        |\u|_{W^{s,p}_{\tau}(\Omega)^n}
        := \left(
          \int_\Omega \int_{B(x,\tau \delta(x))}
          \frac{|\u(y) - \u(x)|^p}{|y - x|^{n+sp}} \dy\dx
        \right)^{1/p}
      \end{equation}\
      where $\tau\in (0,1)$ and $\delta(x)=\dist(x,\partial\Omega)$.
      Indeed, one always has, for every $\tau\in (0,1)$ and $1<p<\infty$, that $|\u|_{W^{s,p}_\tau(\O)^n}\le C|\u|_{W^{1,p}(\O)^n}$  \cite[Lemma 2.2]{DD2022}.  A deeper result states that, in domains like the one in the above example, the induced truncated space $W^{s,p}_\tau(\O)^n$ is actually a real interpolation space between $L^p(\O)^n$ and $W^{1,p}(\O)^n$ \cite{DD2019}.

      In analogy with \eqref{eq:rest_gagliardo_semi} we also introduce the truncated seminorm given by 
      \begin{equation}
        \label{eq:rest_meng_semi}
        |\u|_{X^{s,p}_{\tau}(\Omega)^n}
        := \left(
        \int_\Omega \int_{B(x,\tau \delta(x))} \frac{|(\u(y) - \u(x)) \cdot (y - x)|^p}{|y - x|^{n+sp+p}} \dx\dy \right)^{1/p},
      \end{equation}\
      where $\tau\in (0,1)$ and $\delta(x)=\dist(x,\partial\Omega)$, and consider fractional Korn inqualities involving these seminorms.

      Given a bounded John domain $\Omega\subset\R^n$ and a Whitney decomposition $\{Q_t\}_{t\in\Gamma}$ as described in Section \ref{section:domains_trees}, we assume $\Gamma$ is chosen as in Lemma \ref{lemma:treeJohn}. Let us denote $\ell_t$ the edge-length of $Q_t$ and $q_t$ its center, so that $Q_t = \tfrac{\ell_t}{2}\hat{Q} + q_t$. We consider two expansions of $Q_t$. First, we define
      \begin{align*}
        \hat{Q}^{*/2}:=& \{x\in\R^n:\,\|x\|_\infty \leq 1+\tfrac{1}{4}\}\\
        \hat{Q}^*:    =& \{x\in\R^n:\,\|x\|_\infty \leq 1+\tfrac{1}{2}\},
      \end{align*}
      which in turn allows us to introduce
      \begin{align*}
        Q_t^{*/2}:=& \tfrac{\ell_t}{2}\hat{Q}^{*/2}+q_t,\\
        Q_t^*:    =& \tfrac{\ell_t}{2}\hat{Q}^{*}+q_t.
      \end{align*}
      These cubes have the same center as $Q_t$, but an edge-length expanded by a factor of $\tfrac{1}{4}$ and $\tfrac{1}{2}$ respectively.

      Now, we can fix a value of $N$ for the definition of the reference smoothened cube $\hat{U}$ given in \eqref{eq:general_smoothen}. In particular, we choose $N$ large enough so that:
      \[\hat{Q}^{*/2}\subset (1+\tfrac{1}{2})\hat{U} \subset \hat{Q}^*.\]
    Actually, any $N\geq \tfrac{\ln(n)}{\ln(6)-\ln(5)}$ works, where the comparison between $\hat{U}$ and the expanded cubes of $\hat{Q}$ follows from the relation between $\|\cdot\|_N$ and $\|\cdot\|_\infty$:
    \[
      \|x\|_\infty \le \|x\|_N \le n^{1/N} \|x\|_\infty.
    \]

    We will work with the following expanded and smoothened cubes based on $Q_t$:
    \begin{equation}\label{def:smoothen_cubes}
      U_t := (1+\tfrac{1}{2})\tfrac{\ell_t}{2}\hat{U}+q_t.
    \end{equation}
    Notice that this implies
    \[Q^{*/2}_t \subset U_t \subset Q^*_t.\]

    In the sequel we will use extensively that, since they are small expansions of Whitney cubes, the cubes $Q_t^*$ have finite overlapping. In particular, $\sum_{t\in\Gamma}\chi_{Q_t^*}(x)\leq C$, for some constant $C$ depending only on the dimension $n$.   %

    The following will be our local result:
    \begin{lemma}\label{lemma:Menguesha_smoothen_cubes}
      There exists a constant $C\geq 1$ such that
      \begin{equation*}
        \inf_{\r \in RM} |\u-\r|_{W^{s,p}(U_t)^n} \le C |\u|_{X^{s,p}(U_t)^n},
      \end{equation*}
      for any vector field $\u\in W^{s,p}(U_t)^n$ and any smoothened cube in the collection defined in \eqref{def:smoothen_cubes}.
    \end{lemma}
    \begin{proof}
      In \cite[Theorem 1.3]{Harutyunyan2023a}, the authors established the validity of the fractional Korn inequality for domains with smooth boundaries. Hence, the inequality holds for each smoothened cube \( U_t \). In this lemma, we show that the constant valid for \( \hat{U} \) also ensures the validity of the inequality for the remaining smoothened cubes \( U_t \) via a change of variables.

      Let us consider \( U = F(\hat{U}) \), \(x = F(\hx) \), and \( y = F(\hy) \), where \( F(\hx) = \tfrac{\ell}{2} \hx + q \) where $\ell$ is the edge-length of $U$ and $q$ its center. Thus, given \( \v \in W^{s,p}(U)^n \)     %
      \begin{align*}
        |\v|^p_{W^{s,p}(U)^n} &=  \int_U \int_U \frac{|\v(x) - \v(y)|^p}{|x - y|^{n+sp}} \,\dy\dx\\
        &=  \int_{\hat{U}} \int_{\hat{U}} \frac{|{\v}(F(\hx)) - {\v}(F(\hy))|^p}{|F(\hx) - F(\hy)|^{n+sp}} \left(\tfrac{\ell}{2}\right)^{2n} \,\dhy\,\dhx\\
        &= \left(\tfrac{\ell}{2}\right)^{n-sp}  \int_{\hat{U}} \int_{\hat{U}} \frac{|{\v}(F(\hx)) - {\v}(F(\hy)|^p}{|\hx - \hy|^{n+sp}} \,\dhy\,\dhx
          =  \left(\tfrac{\ell}{2}\right)^{n-sp}\, |{\v}\circ F|^p_{W^{s,p}(\hat{U})^n}.\\
          \intertext{On the other hand, taking into account that $y-x = \tfrac{\ell}{2}(\hy-\hx)$ and $|y-x| = \left(\tfrac{\ell}{2}\right)^n|\hy-\hx|$ and applying the same change of variables we obtain}
          |{\v}|^p_{X^{s,p}(U)^n} &=  \int_U \int_U \frac{|(\v(y) - \v(x)) \cdot (y - x)|^p}{|y - x|^{n+sp+p}} \dy\dx\\
          &=  \int_{\hat{U}} \int_{\hat{U}} \frac{\left(\tfrac{\ell}{2}\right)^p|(\v(F(\hy)) - \v(F(\hx))) \cdot (\hy - \hx)|^p}{\left(\tfrac{\ell}{2}\right)^{n+sp+p}|\hy - \hx|^{n+sp+p}} \,\dhy\,\dhx\\
          &= \left(\tfrac{\ell}{2}\right)^{n-sp}\, |{\v}\circ F|^p_{X^{s,p}(\hat{U})^n}.
        \end{align*}
        Next, using the fact that \( \r \mapsto \r \circ F \) is a bijection on the space of infinitesimal rigid motions \( RM \), together with the validity of the fractional Korn inequality on \( \hat{U} \) with constant \( C_0 \), and the identities established above, we can conclude that
        \begin{align*}
          \inf_{\r\in RM} |{\v}-\r|_{W^{s,p}(U)^n} &= \left(\tfrac{\ell}{2}\right)^{n/p-s}\, \inf_{\r\in RM} |{\v}\circ F-\r|_{W^{s,p}(\hat{U})^n}\\
          &\leq C_0\,\left(\tfrac{\ell}{2}\right)^{n/p-s}\, |{\v}\circ F|_{X^{s,p}(\hat{U})^n}= C_0\, |{\v}|_{X^{s,p}(U)^n}.
        \end{align*}
        Hence, the proof is complete.
      \end{proof}

      The following theorem is a \emph{truncated} fractional Korn inequality for John domains. By \emph{truncated} we mean that both norms are truncated as in \eqref{eq:rest_gagliardo_semi} and \eqref{eq:rest_meng_semi}. In the proof we need to apply the discrete Poincaré inequality \eqref{eq:discretePoincare}. The validity of such inequality is proven later in Lemma \ref{lemma:discretePoincareholds}, under a very general setting. Since it is a rather technical result, we difer its proof.

      \begin{theorem}\label{thm:rest_fractional_Korn_John}
        Let $\Omega\subset\R^n$ be a bounded John domain, $1<p<\infty$, and $0<s<1$. Take $\tau_1<\frac{1}{36\sqrt{n}}$ and $\frac{3}{5}\le \tau_2 < 1$. Then, there exists a constant $C$ such that
        \begin{equation}\label{eq:restricted_fractional_Korn}
          \inf_{\r \in RM}|\u-\r|_{W^{s,p}_{\tau_1}(\Omega)^n} \le C |\u|_{\chi^{s,p}_{\tau_2}(\Omega)^n},
        \end{equation}
        for any vector field $\u\in W^{s,p}_{\tau_2}(\Omega)^n $.
      \end{theorem}

      \begin{proof}
        %Let us denote by $A$ a skew-symmetric matrix to be chosen below, and consider the vector field ${\bm w}(x) = b+Ax$ where the constant vector $b$ is such that $\int_\Omega (\u-{\bm w}) = 0$. Then, using that $\Omega$ is uniform and applying the (improved) Poincaré inequality we have

        Let $\r(x) = b + Ax \in RM$. Since we are studying only the seminorm, the constant $b$ is immaterial. We will choose a particular skew-symmetric matrix $A$ below.

        \begin{align*}
          |\u-\r|_{W^{s,p}_{\tau_1}(\Omega)^n}^p &= \int_\Omega\int_{|y-x|<\tau_1 \delta(x)} \frac{|\u(x)-\u(y) - A(x-y)|^p}{|x-y|^{n+ps}}\dy\dx\\
          &=  \sum_{t\in V}\int_{Q_t}\int_{|y-x|<\tau_1 \delta(x)} \frac{|\u(x)-\u(y) - A(x-y)|^p}{|x-y|^{n+ps}}\dy\dx.
        \end{align*}

        Now, observe that $x\in Q_t$, so $\delta(x) \le d(Q_t,\partial\Omega)+\diam(Q_t)\le 9\diam(Q_t)$. Hence
        \[|y-x| < \tau_1 \delta(x) \le \frac{9\diam(Q_t)}{36\sqrt{n}} = \frac{\ell_t}{4},\]
        which implies that $y \in Q_t^{*/2} \subset U_t$, so we continue by extending the domains of integration:
        \begin{align*}
          |\u-\r|_{W^{s,p}_{\tau_1}(\Omega)^n}^p &\le \sum_{t\in V}\int_{Q_t}\int_{U_t} \frac{|\u(x)-\u(y) - A(x-y)|^p}{|x-y|^{n+ps}}\dy\dx\\
          &\le \sum_{t\in V}\int_{U_t}\int_{U_t} \frac{|\u(x)-\u(y) - A(x-y)|^p}{|x-y|^{n+ps}}\dy\dx.
          \intertext{In order to estimate this, consider $A_t = \Pi_{U_t}\u$, the local projection of the $s$ gradient of $\u$ onto the skew-symmetric matrices. Then}
          &\lesssim \sum_{t\in V}\int_{U_t}\int_{U_t}\frac{|\u(x)-\u(y) - A_t(x-y)|^p}{|x-y|^{n+ps}}\dy\dx\\ &\quad + \sum_{t\in V}\int_{U_t}\int_{U_t}\frac{|(A_t-A)(x-y)|^p}{|x-y|^{n+ps}}\dy\dx\\
          &= I + II
        \end{align*}

        $I$ can be estimated by Lemma \ref{lemma:Menguesha_smoothen_cubes}
        \begin{equation}\label{eq:estimateI}
          I \le C \sum_{t\in V} \int_{U_t}\int_{U_t} \frac{ |(\u(x)-\u(y))\cdot (x-y)|^p}{|x-y|^{n+ps+p}} \dy\dx.
        \end{equation}

        Let us now consider $II$. If we denote
        \[A_t = \sum_{1\le i<j \le n}a_{ij}^t I_{ij}\quad \textrm{and}\quad A=\sum_{1\le i<j \le n}a_{ij} I_{ij},\]
        we have that
        \[A_t-A = \sum_{1\le i< j\le n}(a_{ij}^t-a_{ij})I_{ij}.\]
        This gives:

        \begin{equation}\label{eq:Iijp}
          II \le C \sum_{t\in V}\sum_{1\le i<j\le n} |a_{ij}^t - a_{ij}|^p \int_{U_t}\int_{U_t} \frac{|I_{ij}(x-y)|^p}{|x-y|^{n+ps}}\dy\dx
        \end{equation}

        Applying Corollary \ref{corol:Iijcubes}, we obtain
        \[II \le C \sum_{t\in V}\sum_{1\le i<j \le n}\ell_t^{n+p-ps}|a_{ij}^t-a_{ij}|^p.\]

        Now we can properly choose the matrix $A$. Indeed, we will take $A$ such that

        \begin{equation}\label{eq:zeromean}
          \sum_{t\in V} \ell_t^{n+p-ps} (a_{ij}^t-a_{ij}) = 0 \quad \forall 1\le i<j\le n.
        \end{equation}
        Note that this implies
        \[ \sum_{t\in V} \ell_t^{n+p-ps} (A_t-A) = 0,\]
        or, in other words
        \[A = \frac{1}{\sum_{t\in V} \ell_t^{n+p-ps}}\sum_{t\in V} \ell_t^{n+p-ps} A_t.\]

        We need to check that this matrix is well defined, i.e.: that the summation is finite. First observe that
        \[\sum_{t\in V}\ell_t^{n+p-ps} \le \sum_{t\in V}\ell_t^n = \sum_{t\in V}|Q_t| = |\Omega|.\]
        On the other hand, for the summation of the matrices $A_t$ it is enough to consider each coefficient, namely:
        \[\sum_{t\in V}\ell_t^{n+p-ps}a_{ij}^t.\]
        Now, thanks to the definition of $A_t$, the H\"older inequality and Lemma \ref{lemma:seminormIijcubes} we have that
        \begin{align*}
          a_{ij}^t &= |\I_{ij}|_{W^{s,2}(U_t)^n}^{-2} \langle \u,\I_{ij}\rangle_{U_t\times U_t} \le |\I_{ij}|_{W^{s,2}(U_t)^n}^{-2} |\u|_{W^{s,p}(U_t)^n}|\I_{ij}|_{W^{s,p'}(U_t)^n} \\
          &= C \ell_t^{-n-2+2s}\ell_t^{\frac{n}{p'}+1-s} |\u|_{W^{s,p}(U_t)^n} = C\ell_t^{-\frac{n}{p}-1+s}|\u|_{W^{s,p}(U_t)^n},
        \end{align*}
        Using this and the H\"older inequality we have
        \begin{align*}
          \sum_{t\in V}\ell_t^{n+p-ps}a_{ij}^t &= C\sum_{t\in V}\ell_t^{\frac{n}{p'}+p-1-(p-1)s} |\u|_{W^{s,p}(U_t)^n} \\
          &\le C \left(\sum_{t\in V} \ell_t^{n+p-ps}\right)^\frac{1}{p'}\left(\sum_{t\in V}|\u|_{W^{s,p}(U_t)^n}^p\right)^\frac{1}{p},
          \intertext{and by the finite overlapping of the smoothened cubes $U_t$ we conclude}
          &\le C|\Omega|^\frac{1}{p'}|\u|_{W^{s,p}(\Omega)^n}.
        \end{align*}

        Now we consider the sequence ${\bf b} = \{b_t\}_{t\in V}$ with $b_t = a_{ij}^t-a_{ij}$. \eqref{eq:zeromean} establishes a zero-mean property for ${\bf b}$, which allows us to apply the discrete Poincaré inequality \eqref{eq:discretePoincare} with $\nu_t=\mu_t=\ell_t^{n+p-ps}$. For this, we need Lemma \ref{lemma:discretePoincareholds} (stated below). Notice that the exponent $\gamma$ in that lemma corresponds here to $1-s$, and condition \eqref{eq:condgamma} follows.

        \begin{align*}
          \sum_{t\in V}\ell_t^{n+p-ps} |a_{ij}^t-a_{ij}|^p &\le C \sum_{t\in V}\ell_t^{n+p-ps} |a_{ij}^t - a_{ij}^{t_p}|^p,
        \end{align*}
        where $t_p$ is the parent of $t$.
        We take $\tilde{U}_t$ a smoothened cube contained in $U_t\cap U_{t_p}$ with edge-length proportional to $\ell_t$. We denote $\tilde{A}_{t}=\Pi_{\tilde{U}_t\times \tilde{U}_t}\u$, with coefficients $\tilde{a}^t_{ij}$, and insert $\tilde{a}_{ij}^{t}$:

        \begin{align*}
          \quad\quad &\lesssim \sum_{t\in V}\ell_t^{n+p-ps} |a_{ij}^t - \tilde{a}_{ij}^{t}|^p + \sum_{t\in V}\ell_t^{n+p-ps} |\tilde{a}_{ij}^{t} - a_{ij}^{t_p}|^p.
        \end{align*}
        It is clear that both terms are analogous, so we only consider the first one. Using \eqref{eq:Iijp} we can recover the integral:

        \begin{align*}
          \sum_{t\in V} \ell_{t}^{n+p-ps}|a_{ij}^t-\tilde{a}_{ij}^{t}|^p &\approx \sum_{t\in V} \int_{\tilde{U}_t}\int_{\tilde{U}_t}|a_{ij}^t-\tilde{a}_{ij}^{t}|^p\frac{|I_{ij}(x-y)|^p}{|x-y|^{n+p-ps}}\dy\dx.
        \end{align*}
        And summing over $i,j$:
        \begin{align*}
          \sum_{1\le i<j\le n}\sum_{t\in V} \ell_{t}^{n+p-ps}|a_{ij}^t-\tilde{a}_{ij}^{t}|^p &\lesssim \sum_{t\in V} \int_{\tilde{U}_t}\int_{\tilde{U}_t}\frac{|(A_t-\tilde{A}_{t})(x-y)|^p}{|x-y|^{n+p-ps}}\dy\dx.
        \end{align*}

        Here we can insert $\u$, apply the triangular inequality and expand the domain of integration to obtain

        \begin{align*}
          \quad\quad &\lesssim \sum_{t\in V} \int_{\tilde{U}_t}\int_{\tilde{U}_t}\frac{|\u(x)-\u(y)-A_t(x-y)|^p}{|x-y|^{n+p-ps}}\dy\dx\\
          &\quad\quad + \sum_{t\in V} \int_{\tilde{U}_t}\int_{\tilde{U}_t}\frac{|\u(x)-\u(y)-\tilde{A}_t(x-y)|^p}{|x-y|^{n+p-ps}}\dy\dx.\\
          &\le \sum_{t\in V} \int_{U_{t}}\int_{U_{t}}\frac{|\u(x)-\u(y)-A_t(x-y)|^p}{|x-y|^{n+p-ps}}\dy\dx\\
          &\quad\quad + \sum_{t\in V} \int_{\tilde{U}_t}\int_{\tilde{U}_t}\frac{|\u(x)-\u(y)-\tilde{A}_t(x-y)|^p}{|x-y|^{n+p-ps}}\dy\dx.
        \end{align*}

        In these terms we can apply Lemma \ref{lemma:Menguesha_smoothen_cubes}, and the finite overlapping of $U_t$ and $\tilde{U_t}$, obtaining
        \[II\le C \sum_{t\in V} \int_{U_t}\int_{U_t} \frac{\left|(\u(x)-\u(y))\cdot (x-y)\right|^p}{|x-y|^{n+sp+p}}\dy\dx,\]
        which is the same estimate that we previously obtained for $I$ in \eqref{eq:estimateI}. In order to finish the proof, let us observe that, since $x\in U_t$,
        \begin{align*}
          \tau_2 \delta(x) &\ge \frac{3}{5} d(U_t,\partial\Omega) \ge \frac{3}{5}(d(Q^*_t,\partial\Omega)-\frac{1}{2}diam(Q_t))\\
          &\ge \frac{3}{5}(3\diam(Q_t)-\frac{1}{2}\diam(Q_t)) = \frac{3}{2}\diam(Q_t) \\
          &= \diam(Q^*_t),
        \end{align*}
        and consequently, for $x\in U_t$, we have that $U_t\subset Q_t^* \subset \{y:|y-x|<\tau_2 \delta(x)\}$, which gives
        \begin{align*}
          I+II &\le C \sum_{t\in V} \int_{U_t}\int_{U_t} \frac{\left|(\u(x)-\u(y))\cdot(x-y)\right|^p}{|x-y|^{n+sp+p}}\dy\dx \\
          &\le C \sum_{t\in V} \int_{U_t}\int_{|y-x|<\tau_2 \delta(x)} \\
          &\le C \int_\Omega \int_{|y-x|<\tau_2 \delta(x)} \frac{\left|(\u(x)-\u(y))\cdot(x-y)\right|^p}{|x-y|^{n+sp+p}}\dy\dx,
        \end{align*}
        which completes the proof.
      \end{proof}

      \begin{remark}
        The conditions $\tau_1<\frac{1}{36\sqrt{n}}$ and $\tau_2 \ge \frac{3}{5}$ are imposed in order to be able to pass from $\{y:\,|y-x|<\tau_1 d(x)\}$ to $U_t$ and then from $U_t$ to $\{y:\,|y-x|<\tau_2 d(x)\}$. Hence they can be slightly relaxed by adapting the Whitney partition accordingly.
      \end{remark}

      Let us now prove the validity of the discrete Poincaré inequality on trees based on John domains as in Lemma \ref{lemma:treeJohn}. We state the result for weights of the form $\ell_t^{n+p\gamma}$. In the previous theorem, $\gamma=1-s$. The following lemma imposes a condition on the exponent $\gamma$, in terms of the Assouad dimension of the boundary, so we begin by introducing this dimension.

      \begin{definition}
        Given a set $E\subset \R^n$ and $r>0$, we denote by $N_r(E)$ the least number of open balls of radius $r$ that are needed for covering $E$. The \emph{Assouad dimension} of $E$, denoted $\dim_A(E)$ is the infimal $\lambda\geq 0$ for which there is a positive constant $C$, such that for every $0<r<R$ the following estimate holds:
        \[\sup_{x\in E} N_r(B_R(x)\cap E) \le C\bigg(\frac{R}{r}\bigg)^\lambda.\]
      \end{definition}

      The Assouad dimension is the largest of the commonly used notions of dimension. As an example, consider the set in $\R$, $\mathcal{S} = \{0\}\cup\{\tfrac{1}{n}:\,n\in\mathbb{N}\}$. Since $\mathcal{S}$ is a numerable set, its Hausdorff dimension is zero. On the other extreme, $\dim_A(\mathcal{S})=1$, because the Assouad dimension captures the local behaviour at the accumulation point as if it were a line. In the middle, the box dimension of $\mathcal{S}$ is one half. We refer the reader to \cite{Fraser} for more details on the Assouad dimension. It is worth noting, however, that the Assoud dimension has been associated with distance weights. For example, in \cite{Dyda19}, a characterization of the exponents $\beta$ for which $d(\cdot,F)^\beta$ belongs to the Muckenhoupt class $A_p$ is given in terms of the Assouad codimension of the set $F$. Our results are based on estimates obtained in \cite{LGO-2025}, where several weighted inequalities are proven on John domains, with weights given by powers of the distance to the boundary of the domain satisfying a restriction similar to \eqref{eq:condbeta}.

      \begin{lemma}\label{lemma:discretePoincareholds} Let $\Omega \subset \R^n$ a bounded John domain, $\{Q_t\}_{t\in V}$ a Whitney decomposition and $\ell_t$ the edge-length of $Q_t$. Moreover take weights $\mu_t = \nu_t = \ell_t^{n+p\gamma}$, with $\gamma\in\R$ satisying
        \begin{equation}\label{eq:condgamma}
          \gamma > -(n-\dim_A(\partial\Omega))/p.
        \end{equation}
        Then the discrete weighted Poincaré inequality \eqref{eq:discretePoincare} holds with weights $\mu_t$ and $\nu_t$.
      \end{lemma}
      \begin{proof}
        Thanks to Theorem \ref{thm:suff cond} it is enough to verify \eqref{eq:suffcond} for a John domain and the desired weights. The condition involves three vertices $t$, $u$, and $v$. For clarity we consistently apply the notation $\ell_t = 2^{-k}$, $\ell_u=2^{-i}$ and $\ell_v=2^{-j}$, for $i,j,k \in \mathbb{N}$. Furthermore,  we assume that the largest Whitney cube in $\Omega$ has edge-length $1$. Let us denote
        \begin{equation}\label{eq:P(i)}
          \mathbb{P}_i(t) = \#\{u\in V:\, u\preceq t, \ell_u = 2^{-i}\},
        \end{equation}
        \begin{equation}\label{eq:W(i)}
          \mathbb{W}_i(t) = \#\{u\in V:\, u\succeq t, \ell_u = 2^{-i}\}.
        \end{equation}
        $\mathbb{P}_i(t)$ is the number of cubes of size $2^{-i}$ in the chain that connects the root of the tree, $t_0$ with $t$, whereas $\mathbb{W}_i(t)$ is the number of cubes of size $2^{-i}$ in the \emph{shadow} of $t$ (recall \eqref{eq:shadow}).

        The definition of John domain implies that $\mathbb{P}_i(t) \le K^n$, where $K$ is the constant given in Lemma \ref{lemma:treeJohn} which depends only on $\Omega$. On the other hand, we have that $\mathbb{W}_i(t)\le C \left(\tfrac{2^i}{2^k}\right)^\lambda$ for every $\lambda > \dim_A(\partial\Omega)$. This was proven in \cite[Lemma 4.5]{LGO-2025}.

        In order to estimate the summations for $u\preceq t$ and $u \succeq t$, it is important to notice that the cubes $Q_u$ such that $u\preceq t$ are tipically larger than $Q_t$, but can eventually be a little smaller. Similarly, almost all cubes $Q_u$ such that $u\succeq t$ are smaller than $Q_t$, but some of them can be slightly larger. In general, we can choose a fixed constant $M$ depending only on $\Omega$ such that $\ell_u \ge 2^{-k-M}$ for every $u\preceq t$ and $\ell_u \le 2^{-k+M}$ for every $u\succeq t$. Let us now consider the first factor in \eqref{eq:suffcond}. We have:
        \begin{align*}
          \sum_{t_0\prec u \preceq t}\mu_u^{-\frac{p'}{p}} &= \sum_{t_0\prec u \preceq t}\ell_u^{-n\frac{p'}{p}-\gamma p'} = \sum_{j=0}^{k+M} \mathbb{P}_j(t) (2^{-j})^{-n\frac{p'}{p}-\gamma p'} \\
          &\le C (2^{-(k+M)})^{-n\frac{p'}{p}-\gamma p'} \le C \ell_t^{-n\frac{p'}{p}-\gamma p'}.
        \end{align*}

        In the second factor we have the inner summation that can be estimated in the same way:
        \begin{align*}
          \sum_{t_0\prec v\preceq u} \mu_v^{-\frac{p'}{p}} \le C \ell_u^{-n\frac{p'}{p}-\gamma p'}.
        \end{align*}

        Inserting this in the second factor we obtain:
        \begin{align*}
          \sum_{u \succeq t}\nu_u \left(\sum_{t_0\prec v \preceq u} \mu_v^{-\frac{p'}{p}}\right)^{\frac{p}{p'}(1-\frac{1}{\theta})} &\le C \sum_{u\succeq t} \ell_u^{n+p\gamma} (\ell_u^{-n\frac{p'}{p}-\gamma p'})^{\frac{p}{p'}(1-\frac{1}{\theta})} = C \sum_{u\succeq t} \ell_u^{\frac{1}{\theta}(n+p\gamma)} \\
          &= C \sum_{i=k-M}^\infty \mathbb{W}_i(t) (2^{-i})^{\frac{1}{\theta}(n+p\gamma)}\\
          &\le C \sum_{i=k-M}^\infty \left(\frac{2^i}{2^k}\right)^\lambda (2^{-i})^{\frac{1}{\theta}(n+p\gamma)}\\
          &\le C 2^{-k\lambda} \sum_{i=k-M}^\infty (2^{-i})^{-\lambda+\frac{1}{\theta}(n+p\gamma)}.
          \intertext{Now, we use condition \eqref{eq:condgamma}, which implies that there is a $\lambda$ close enough to $\dim_A(\partial\Omega)$ and $\theta$ close enough to $1$ such that $-\lambda + \frac{1}{\theta}(n+p\gamma) >0$, then can write}
          &\le C 2^{-k\lambda} (2^{-(k-M)})^{-\lambda+\frac{1}{\theta}(n+p\gamma)} \\
          &= C 2^{-k\lambda + k\lambda -\frac{k}{\theta}(n+p\gamma)} = C \ell_t^{\frac{1}{\theta}(n+p\gamma)}.
        \end{align*}

        Joining the estimates for the two factors we have:

        \begin{align*}
          C_{tree} \le C \sup_{t\in \Gamma}(\ell_t^{-n\frac{p'}{p}-p'\gamma})^{\frac{1}{\theta p'}} (\ell_t^{\frac{1}{\theta}(n+p\gamma)})^{\frac{1}{p}}\le C,
        \end{align*}
        which concludes the proof.
      \end{proof}

      Finally, we can state a weighted analogue of Theorem \ref{thm:rest_fractional_Korn_John}. To this end we introduce the weighted seminorms

      \begin{equation}
        \label{eq:rest_gagliardo_semi_pesos}
        |\u|_{W^{s,p}_{\tau}(\Omega, d^\alpha)^n}
        := \left(
          \int_\Omega \int_{B(x,\tau \delta(x))}
          \frac{|\u(y) - \u(x)|^p}{|y - x|^{n+sp}} \, d(x,y)^\alpha \dy\dx
        \right)^{1/p}
      \end{equation}\
      and 

      \begin{equation}
        \label{eq:rest_meng_semi_pesos}
        |\u|_{X^{s,p}_{\tau}(\Omega, d^\alpha)^n}
        := \left(
        \int_\Omega \int_{B(x,\tau \delta(x))} \frac{|(\u(y) - \u(x)) \cdot (y - x)|^p}{|y - x|^{n+sp+p}} \, d(x,y)^\alpha \dx\dy \right)^{1/p},
      \end{equation}\
      where $\tau\in (0,1)$, $\delta(x)=\dist(x,\partial\Omega)$, and $d(x,y)=\min\{\delta(x),\delta(y)\}$.

      \begin{theorem}\label{thm:rest_fractional_weighted_Korn_John}
        Let $\Omega\subset\R^n$ be a bounded John domain, $1<p<\infty$, and $0<s<1$. Take $\tau_1<\frac{1}{36\sqrt{n}}$ and $\frac{3}{5}\le \tau_2 < 1$. If $\beta\in \R$ is such that
        \begin{equation}\label{eq:condbeta}
          \beta + 1 - s > -(n-\dim_A(\partial\Omega))/p,
        \end{equation}
        then, there exists a constant $C$ such that
        \begin{equation}\label{eq:restricted_fractional_weighted_Korn_John}
          \inf_{\r \in RM}|\u-\r|_{W^{s,p}_{\tau_1}(\Omega,d^{p\beta })^n} \le C |\u|_{X^{s,p}_{\tau_2}(\Omega,d^{p\beta })^n},
        \end{equation}
        for any vector field $\u\in W^{s,p}_{\tau_2}(\Omega,d^{p \beta })^n $.
      \end{theorem}
      \begin{proof}
        The proof is essentially the same as that of Theorem \ref{thm:rest_fractional_Korn_John}. We only add some details where the weight $d^\beta$ plays a relevant role. We begin in the same way, localizing the norm in $U_t\times U_t$. Then, we pull the weight out of the integrals, by using that $d(x,y)^{p\beta}\approx \ell_t^{p\beta}$.

        \begin{align*}
          |\u-\r|^p_{W^{s,p}_{\tau_1}(\Omega,d^{p\beta})^n} &\le C\sum_{t\in V}\int_{U_t}\int_{U_t}\frac{|\u(x)-\u(y)|^p}{|x-y|^{n+ps}}d(x,y)^{p\beta} \dy\dx\\
          &\le C\sum_{t\in V}\ell_t^{p\beta}\int_{U_t}\int_{U_t}\frac{|\u(x)-\u(y)|^p}{|x-y|^{n+ps}} \dy\dx.
        \end{align*}

        Now, the proof continues exactly as in Theorem \ref{thm:rest_fractional_Korn_John}. For the term $I$, the (unweighted) local inequality given by Lemma \ref{lemma:Menguesha_smoothen_cubes} is applied. However, the factor $\ell_t^{p\beta}$ needs to be taken into account when dealing with the term $II$. In particular, the matrix $A$ is defined as
        \[A = \frac{1}{\sum_{t\in V}\ell_t^{n+p-ps+p\beta}} \sum_{t\in V}\ell_t^{n+p-ps+p\beta}A_t\] The good definition of $A$ relies on the boundedness of $\sum_{t\in V}\ell_t^{n+p-ps+p\beta}$. In order to estimate this summation, let us recall \eqref{eq:W(i)} and the fact that $\mathbb{W}_i(t_0) \le C \left(\tfrac{2^i}{2^{k_0}}\right)^\lambda$ where $\ell_{t_0}= 2^{-k_0}$ and $\lambda>\dim_A(\partial\Omega)$. Then:
        \begin{align*}
          \sum_{t\in V} \ell_t^{n+p-ps+p\beta} &= \sum_{t\succeq t_0} \ell_t^{n+p-ps+p\beta} \\
          &= \sum_{i=k_0}^\infty \mathbb{W}_i(t_0) (2^{-i})^{n+p-ps+p\beta} \\
          &\le C\sum_{i=k_0}^\infty \left(\tfrac{2^i}{2^{k_0}}\right)^\lambda (2^{-i})^{n+p-ps+p\beta} \\
          &\le C \sum_{i=k_0}^\infty (2^{-i})^{n+p-ps+p\beta -\lambda},
        \end{align*}
        which is finite since condition \eqref{eq:condbeta} implies $n+p-ps+p\beta-\lambda>0$.

        Finally, the discrete Poincaré inequality can be applied with weights $\mu_t=\nu_t = \ell_t^{n+p-ps+p\beta}$. Indeed, in terms of Lemma \ref{lemma:discretePoincareholds} we have $\gamma = \beta + 1 -s$, and condition \eqref{eq:condgamma} translates exactly into \eqref{eq:condbeta}.

        The rest of the analysis is exactly as in Theorem \ref{thm:rest_fractional_Korn_John}: the unweighted inequality given by Lemma \ref{lemma:Menguesha_smoothen_cubes} is applied on $U_t$ and on $\tilde{U}_t$, obtaining
        \begin{align*}
          I+II &\le C\sum_{t\in V}\ell_t^{p\beta} \int_{U_t}\int_{|y-x|<\tau_2 \delta(x)} \frac{|(\u(x)-\u(y))\cdot(y-x)|^p}{|x-y|^{n+sp+p}}\dy\dx\\
          &\le C\sum_{t\in V}\int_{U_t}\int_{|y-x|<\tau_2 \delta(x)} \frac{|(\u(x)-\u(y))\cdot(y-x)|^p}{|x-y|^{n+sp+p}}d(x,y)^{p\beta}\dy\dx\\
          &\le C |\u|_{X^{s,p}_{\tau_2}(\Omega,d^{p\beta})^n}^p,
        \end{align*}
        where we used that for $x\in U_t$, and $|y-x|<\tau_2 \delta(x)$, one has $d(x,y)\approx \ell_t$, and the finite overlapping of the smoothened cubes $U_t$.
      \end{proof}

      \section{Fractional Korn inequality on uniform domains}

      In this section we study the fractional Korn inequality on uniform domains.

      \begin{theorem}\label{thm:fractional_Korn_uniform}
        Let $\Omega\subset\R^n$ be a bounded uniform domain, $1<p<\infty$, and $0<s<1$. Then, there exists a constant $C>1$ such that
        \begin{equation}\label{eq:fractional_Korn}
          \inf_{r \in RM} |\u-\r|_{W^{s,p}(\Omega)^n} \le C |\u|_{X^{s,p}(\Omega)^n},
        \end{equation}
        for any vector field $\u\in W^{s,p}(\Omega)^n $.
      \end{theorem}
      \begin{proof}
        This result follows from Theorem \ref{thm:rest_fractional_Korn_John} once we establish  the equivalence between the fractional seminorm $W^{s,p}(\Omega)^n$ and the truncated fractional seminorm $W^{s,p}_\tau(\Omega)^n$, that holds on uniform domains for every $1<p<\infty$, $0<s<1$, and $0<\tau<1$.
        Indeed, this equivalence follows from the one between the standard and truncated fractional norms \cite[Theorem 1.6]{PratsSaksman}, and the fractional Poincar\'e inequality \cite[Theorem 4.3]{RV2013}
        \begin{equation}
        \|\u - \u_\O\|_{L^p(\O)^n} \le C |\u|_{W^{s,p}_\tau(\O)^n}
        \end{equation}
        where $\u_\O$ is the average of $\u$ over $\O$.

       Now, pick any $\tau_1$ as in the statement of Theorem \ref{thm:rest_fractional_Korn_John} and let $\r(x)=b+Ax$ with $A$ as in the proof of that theorem and $b$ such that $(\u -\r)$ has zero average. Then, we can write
       \begin{align*}
       |\u-\r|_{W^{s,p}(\O)^n}&\le  \|\u-\r\|_{W^{s,p}(\O)^n} \\
       &\le C \|\u-\r\|_{W^{s,p}_{\tau_1}(\O)^n}\\
       &\le C |\u-\r|_{W^{s,p}_{\tau_1}(\O)^n} \\
       &\le C |\u|_{X^{s,p}(\O)^n}
       \end{align*}
       which finishes the proof.
      \end{proof}

\section*{Acknowledgements}

The author L\'opez-Garc\'ia expresses his sincere gratitude to the Departamento de Matem\'atica of the Universidad de Buenos Aires and the Instituto de Investigaciones Matem\'aticas Luis A. Santal\'o (IMAS, CONICET–UBA) for their hospitality during his sabbatical visit.

      \bibliographystyle{plain}
      \bibliography{references.bib}
      \end{document}